\documentclass[10pt,reqno]{amsart}
\usepackage{amsmath,amsthm,amsfonts,amssymb,latexsym, mathtools, multicol,xcolor,enumerate, enumitem,hyperref,cleveref}
\DeclareMathOperator{\spn}{span}

\def\F{\mathbb{F}}

\def\su{{\subseteq}}

\def\dl{{\rm dl}}

\newtheorem{defi}{Definition}[section]
\newtheorem{thm}[defi]{Theorem}
\newtheorem{lem}[defi]{Lemma}
\newtheorem{cor}[defi]{Corollary}

\newtheorem{rem}[defi]{Remark}

\crefrangeformat{equation}{#3(#1)#4--#5(#2)#6}

\crefname{thm}{Theorem}{Theorems}
\crefname{lem}{Lemma}{Lemmas}
\crefname{cor}{Corollary}{Corollaries}
\crefname{rem}{Remark}{Remarks}
\crefname{section}{Section}{Sections}
\crefname{equation}{\unskip}{\unskip}
\crefname{enumi}{\unskip}{\unskip}

\newcommand{\ch}[1]{\mathrm{char}{(#1)}}

\newcommand{\U}{\mathcal{U}}

\begin{document}
	
	\title[Identities and derived lengths of finitary incidence algebras and their group of units]{Identities and derived lengths of finitary incidence algebras and their group of units}
	
	\author{Mykola Khrypchenko}
	\address{Departamento de Matem\'atica, Universidade Federal de Santa Catarina,  Campus Reitor Jo\~ao David Ferreira Lima, Florian\'opolis, SC, CEP: 88040--900, Brazil \and CMUP, Departamento de Matemática, Faculdade de Ciências, Universidade do Porto,
		Rua do Campo Alegre s/n, 4169--007 Porto, Portugal}
	\email{nskhripchenko@gmail.com}
	
	\author{\textsc{Salvatore Siciliano}}
	\address{Dipartimento di Matematica e Fisica ``Ennio De Giorgi", Universit\`{a} del Salento,
		Via Provinciale Lecce--Arnesano, 73100--Lecce, Italy}
	\email{salvatore.siciliano@unisalento.it}

	\begin{abstract} Let $FI(X,K)$ be the finitary incidence algebra of a poset $X$ over  a field $K$. In this short note we establish when $FI(X,K)$ satisfies a polynomial identity and when its group of units $\U(FI(X,K))$ satisfies a group identity. The Lie derived length of $FI(X,K)$ and the derived length of $\U(FI(X,K))$ are also determined. 
	\end{abstract}

	\date{\today}
	
	\maketitle

	\section*{Introduction}  
	Let $A$ be a unital associative algebra. Then  $A$ can be regarded as a Lie algebra via the Lie bracket defined by $[x,y]=xy-yx$, for all $x,y \in A$. Given two subspaces $U,V\su A$, we put $UV:=\spn\{uv\mid u\in U,v\in V\}$ and $[U,V]:=\spn\{[u,v]\mid u\in U,v\in V\}$.
	The \textit{Lie derived series} of $A$ is defined inductively by 
	\begin{align*}
		A^{[0]}=A\text{ and }A^{[n+1]}=\left[A^{[n]},A^{[n]}\right],\ n\ge 0.  
	\end{align*}
	Moreover, inspired by \cite{J}, we consider the series of associative two-sided ideals of $A$ defined by  
	\begin{align*}
		A^{(0)}=A\text{ and }A^{(n+1)}=\left[A^{(n)},A^{(n)}\right]A,\ n\ge 0. 
	\end{align*}
	The algebra $A$ is said to be \emph{Lie solvable} (respectively, \emph{strongly Lie solvable}) if $A^{[n]}=\{0\}$ (respectively, $A^{(n)}=\{0\}$) for some $n$. In this case,  the minimal $n$ with such a property is called the \emph{Lie derived length} (respectively, \emph{strong Lie derived length}) of $A$ and denoted by $\dl_{Lie}(A)$ (respectively, $\dl^{Lie}(A)$).
	Clearly, strong Lie solvability implies  Lie solvability of $A$ (and $\dl_{Lie}(A)\leq \dl^{Lie}(A)$), but the converse is not true in general. For instance,  the algebra $M_2(K)$ of $2 \times 2$ matrices over $K$ is Lie solvable but not strongly Lie solvable when $\ch K=2$.
	
	The terms of the \emph{descending  central series} of $A$ are defined inductively by 
	\begin{align*}
		\gamma_1(A)=A\text{ and }\gamma_{n+1}(A)=[\gamma_n(A), A],\ n\geq 1.    
	\end{align*}
	We say that $A$ is \emph{Lie nilpotent} if there is an $n$ such that $\gamma_n(A)=\{0\}$, in which case the minimal $n$ such that $\gamma_{n+1}(A)=\{0\}$ is called the \emph{Lie nilpotency class} of $A$.
	
	The systematic study of the influence of the properties of the associated Lie ring $(R,[\ ,\ ])$ on the structure of a ring $(R,\cdot)$   was initiated by Jennings~\cite{J}. %In particular, if $R$ is an algebra over a field of characteristic different from $2$, then Lie solvability and Lie nilpotency of $(R,\cdot)$ are essentially the same as solvability and nilpotency of $(R,[\ ,\ ])$, respectively. 
	Nowadays, there is an extensive literature devoted to the study of the Lie identities of several classes of associative algebras and to the interplay with the corresponding properties of their group of units.  Passi, Passman and Sehgal~\cite{PPS} completely characterized Lie nilpotency and Lie solvability of the group ring $K[G]$ of a group $G$ over a field $K$ of characteristic $p\ge 0$. % in terms of nilpotency and the so-called ``$p$-abelianity'' of $G$. 
	Riley and Shalev~\cite{RS} studied Lie properties (Lie nilpotency, Lie solvability and the $n$-Engel property) of the (restricted) enveloping algebra $u(L)$ of a restricted Lie algebra $L$ over a field of positive characteristic $p$ different from $2$. The case $p=2$ was treated by Siciliano and Usefi in~\cite{SU1} (see also the survey paper~\cite{SU2}). Lie properties of Poisson algebras were investigated both in the general setting~\cite{SU3}, and in the context of the symmetric Poisson algebra of a Lie algebra~\cite{MP,S,SU4}. There is a relation between Lie nilpotency~\cite{D} and Lie solvability~\cite{AS} of a radical ring $R$ and the corresponding properties of its adjoint group $(R,\circ)$. 
	
	Clearly, any Lie solvable associative algebra satisfies a polynomial identity. Passman~\cite{P1} characterized those Lie algebras whose enveloping algebra (universal or restricted) satisfies a polynomial identity. Giambruno, Sehgal and Valenti~\cite{GSV} proved that the group algebra $K[G]$ of a torsion group $G$ over an infinite field $K$ satisfies a polynomial identity whenever its group of units $U(K[G])$ satisfies a group identity. Passman~\cite{P2} complemented the previous result by describing those $G$ such that $U(K[G])$ satisfies a group identity. The results were generalized to arbitrary fields in \cite{Liu99,LP}.
	
	The purpose of this note is to characterize finitary incidence algebras satisfying a polynomial identity. Recall~\cite{KN} that the \emph{finitary incidence algebra} $FI(X,K)$ of a poset $(X,\leq)$ over a field $K$ is the vector space of formal sums of the form
	$$
	\alpha = \sum_{x\leq y} \alpha_{xy}e_{xy},
	$$
	where $x,y\in X$, $\alpha_{xy}\in K$ and $e_{xy}$ is a symbol, such that for any pair $x<y$ there exists
	only a finite number of $x \leq u <v \leq y$ with $\alpha_{uv}\neq 0$. The product in
	$FI(X,K)$ is given by the convolution
	\begin{align}\label{conv-product}
		\alpha\beta = \sum_{x\leq y}\left(\sum_{x\leq z \leq y}\alpha_{xz}\beta_{zy} \right)e_{xy},
	\end{align}
	so that $FI(X,K)$ is associative and has identity element $\delta:=\sum_{x\in X}e_{xx}$. The algebra $FI(X,K)$ coincides with the classical incidence algebra $I(X,K)$ whenever $X$ is locally finite. Incidence algebras were introduced by Rota~\cite{Rota64} and have become an important tool in combinatorics~\cite{Doubilet-Rota-Stanley72,Stanley} and combinatorial topology~\cite{Orlik-Solomon80,Wachs2007}. Recently, special attention has been paid to the associated Lie algebra~\cite{Coll-Mayers-Russoniello20,Coll-Mayers21} and linear maps on it~\cite{FKS3,FKS,FKS2,FKS4,Zhang-Khrypchenko}.

	Feinberg~\cite{F} proved that the incidence algebra of a locally finite quasiordered set $X$ over a field satisfies a polynomial identity if and only of $X$ is bounded. Nachev~\cite{Nachev77} independently obtained a similar result for incidence algebras over commutative rings. Leroux and Sarreill\'e~\cite{Leroux-Sarreille81} characterized path algebras satisfying a polynomial identity. Spiegel~\cite{Spiegel99} described linear homogeneous identities of $I(X,K)$ of minimal degree. Moreover, in~\cite{Sp} he proved that $I(X,K)$ satisfies a polynomial identity if and only if $I(X,K)$ is Lie solvable, and in~\cite{Spiegel2002} he showed that this is also equivalent to the facts that the group of units $\U(I(X,K))$ of $I(X,K)$ satisfies a group identity and that $\U(I(X,K))$ is solvable. On the other hand, nilpotency and solvability of the group $\U_1=\delta+J(FI(X,K))$ were characterized by Dugas, Herden and Rebrovich~\cite{DHR}.
	
	The main result of our note is \cref{teor} which generalizes the results of \cite{F,Nachev77,Sp,Spiegel2002}. As a consequence, we calculate the Lie derived length of $FI(X,K)$ in \cref{cor1}, and the derived length of its group of units $\U(FI(X,K))$ in \cref{cor1,dl(U)-for-K=F_2}. Moreover, in \cref{cornilp} we recover \cite[Theorem 41]{DHR} complementing it with the exact values of the nilpotency class and derived length of $\U_1$. Finally, in \cref{nilpotency-of-FI} we characterize nilpotency of $FI(X,K)$ and $\U(FI(X,K))$.

	%\section{Preliminaries} 

	\section{Identities and derived lengths of $FI(X,K)$ and $\U(FI(X,K))$} 
	
	Given an associative unital algebra $A$, we write $\U(A)$ for the group of units of $A$ and $J(A)$ for the Jacobson radical of $A$. Given a group $G$, we denote by $G^{(n)}$ the $n$-th term  of the derived series of $G$. If $G$ is solvable, then $\dl (G)$ denotes its derived length. Group commutators are denoted by $(g,h)=g^{-1}h^{-1}gh$, for all $g,h\in G$.
	
	\begin{lem}\label{dl(G)-and-dl^Lie(I)}
		Let $A$ be an associative unital algebra, $I$ an ideal of $A$ and $G$ a subgroup of $\U(A)$ contained in $1+I$. Then $G^{(n)}\su 1+I^{(n)}$. In particular, if $I$ is strongly solvable, then $G$ is solvable and $\dl(G) \leq \dl^{Lie}(I)$.
	\end{lem}
	\begin{proof}
		Induction on $n$. The base case $n=0$ is trivial. Assume then $n>0$ and $G^{(n-1)}\subseteq 1+I^{(n-1)}$. For all $g,h \in G^{(n-1)}$ we have $(g,h)=1+[g^{-1},h^{-1}]gh\in 1+I^{(n)}$. Since $1+I^{(n)}$ is a monoid, this yields the desired conclusion, which completes the proof. 
	\end{proof}
	
	Let $(X,\le)$ be a poset. The \emph{length} of a finite chain $C\su X$ is $l(C):=|C|-1$. The \emph{length} of $X$, denoted $l(X)$, is the supremum of $l(C)$ for all finite chains $C\su X$. We say that $X$ is \emph{bounded} if $l(X)<\infty$.
	% there is an integer $n$ such that $\vert C \vert \leq n$ for every chain in $X$. In this case, the minimal $n$ with such a property is called the bound of $X$ and denoted by $l(X)$. 
	Given $x<y$ in $X$, the length of the interval $\{z\in X \vert \, x\leq z \leq y\}$ will be denoted by $l(x,y)$.
	
	% Note that, for all $x\le y$ and $u\le v$ in $X$,  we have
	% \begin{equation}\label{mult}
		% e_{xy}e_{uv}=
		% \begin{cases}
			% e_{xv}, \; \;\; &\textrm{if } y=u, \\
			% 0, \;\;\;        &\textrm{otherwise}. 
			% \end{cases}
		% \end{equation}
	
	For every $k>0$, put
	$$
	Z_k:=\{\alpha\in FI(X,K) \mid \alpha_{xy}=0 \textrm{ if }l(x,y)\leq k-1\}.
	$$
	
	It turns out that the family $(Z_k)_{k>0}$ is a descending filtration of $FI(X,K)$. In particular, $Z_k$ is an ideal of $FI(X,K)$ for every $k>0$: 
	
	\begin{lem}\label{ideals-Z_k} Let $(X, \leq)$ be a poset and $K$ a field. In the notation above, we have:
		\begin{enumerate}
			\item\label{FI^(1)-sst-Z_1} $FI(X,K)^{(1)}\subseteq Z_1=J(FI(X,K))$;
			\item \label{Z_mZ_n-sst-Z_(m+n)} $Z_m\cdot Z_n\subseteq Z_{m+n}$ for all $m,n>0$;
			\item\label{gamma_n(J(FI))}  $\gamma_{n}(J(FI(X,K))) \subseteq Z_n$ for all $n>0$;
			\item\label{FI^(n)-sst-Z_(2^(n-1))}   $FI(X,K)^{(n)}\subseteq Z_{2^{n-1}}$ for all $n>0$;
			\item\label{J^(n)-sst-Z_(2^n)} $J(FI(X,K))^{(n)}\su Z_{2^n}$ for all $n\ge 0$.
		\end{enumerate}
	\end{lem}
	\begin{proof}
		Assertion \cref{FI^(1)-sst-Z_1} is proved in \cite[Corollaries 2 and 3]{KN}.  For \cref{Z_mZ_n-sst-Z_(m+n)} take $\alpha\in Z_m$ and $\beta\in Z_n$. If $l(x,y)\le m+n-1$, then for any $x\le z\le y$ either $l(x,z)\le m-1$, in which case $\alpha_{xz}=0$, or $l(z,y)\le n-1$, in which case $\beta_{zy}=0$. It follows from \cref{conv-product} that $(\alpha\beta)_{xz}=0$, so $\alpha\beta\in Z_{m+n}$. Items \cref{gamma_n(J(FI)),FI^(n)-sst-Z_(2^(n-1)),J^(n)-sst-Z_(2^n)} are proved by induction on $n$ using \cref{FI^(1)-sst-Z_1,Z_mZ_n-sst-Z_(m+n)}. For instance, let us prove \cref{J^(n)-sst-Z_(2^n)}. We have $J(FI(X,K))^{(0)}=Z_1$ by \cref{FI^(1)-sst-Z_1}. Assuming $J(FI(X,K))^{(n-1)}\su Z_{2^{n-1}}$, thanks to \cref{Z_mZ_n-sst-Z_(m+n)} we obtain
			\begin{align*}
				J(FI(X,K))^{(n)}=[J(FI(X,K))^{(n-1)},J(FI(X,K))^{(n-1)}]J(FI(X,K))\su[Z_{2^{n-1}},Z_{2^{n-1}}]Z_1\su Z_{2^n+1}\su Z_{2^n}.
			\end{align*}
	\end{proof}
	
	Let $K\langle x_1, x_2, \ldots \rangle$ denote the free algebra  on noncommuting indeterminates $x_1,x_2,\ldots$ over the field $K$. We recall that an algebra $A$ over $K$ is said to satisfy a \emph{polynomial identity} if there exists a nonzero $f(x_1,x_2,\ldots,x_n)\in K\langle x_1, x_2, \ldots \rangle$  such that $f(a_1,a_2,\ldots,a_n)=0$, for all $a_1,a_2,\ldots,a_n\in A$. Moreover, a group $G$ is said to satisfy a \emph{group
			identity} whenever there exists a nontrivial word $w(x_1,\ldots,x_n)$ in the free group generated
		by $\{x_1,x_2,\ldots \}$ such that $w(g_1,\ldots,g_n)=1$, for all $g_1,\ldots,g_n\in
		G$.
	
	Finitary incidence algebras satisfying a polynomial identities are characterized in the following result:
	
	\begin{thm} \label{teor}
		Let $(X, \leq)$ be a poset and $K$ a field. The following conditions are equivalent:
		\begin{enumerate}
			\item\label{FI-pol-ident} $FI(X,K)$ satisfies a polynomial identity;
			\item\label{DI-Lie-solv} $FI(X,K)$ is Lie solvable;
			\item\label{FI-str-Lie-solv} $FI(X,K)$ is strongly Lie solvable;
			\item\label{U(FI)-poly-ident} $\U(FI(X,K))$ satisfies a group identity;
			\item\label{U(FI)-solv} $\U(FI(X,K))$ is solvable;
			\item\label{X-bounded} $X$ is bounded.
		\end{enumerate}
	\end{thm}
	\begin{proof}
		We first show that \cref{FI-pol-ident} implies \cref{X-bounded}. To this end, we use the idea of the proof of \cite[Theorem 1]{Spiegel99}. Suppose that $FI(X,K)$ satisfies a polynomial identity $f$ of degree $n$. By the multilinearization process (see e.g. \cite[Theorem 1.3.7]{GZ}), we may assume that $f$ is multilinear. Moreover, by possibly multiplying $f$ by new variables, we may assume that $n=2k-1$. Therefore $FI(X,K)$ satisfies an identity of the form
		\begin{equation}\label{multilin}
			f(\chi_1,\chi_2,\ldots , \chi_{2k-1})=\sum_{\sigma \in \mathrm{Sym}_{2k-1}} c_\sigma \chi_{\sigma(1)}\chi_{\sigma(2)}\cdots \chi_{\sigma(2k-1)}
		\end{equation}
		with coefficients $c_\sigma \in K$, where we can assume that $c_1\neq 0$ (here $\mathrm{Sym}_{2k-1}$ denotes the symmetric group of $\{1,\dots,2k-1\}$). Suppose that there exists a chain $x_1 < x_2 < \ldots <x_d$ in $X$ with $d\geq k$. We must have
		$$
		f(e_{x_1x_1},e_{x_1x_2},e_{x_2x_2},\ldots,e_{x_{k-1}x_k},e_{x_kx_k})=c_1e_{x_1x_k}\neq 0,
		$$
		as the only monomial in \cref{multilin} with non-zero evaluation is $\chi_1 \chi_2 \cdots \chi_{2k-1}$. Thus, $f$ is not a polynomial identity for $FI(X,K)$, a contradiction. This shows that $X$ is bounded with $l(X)< k-1$.
		
		\cref{ideals-Z_k}\cref{FI^(n)-sst-Z_(2^(n-1))} shows that \cref{X-bounded} implies \cref{FI-str-Lie-solv}, as $FI(X,K)^{(n)}=\{0\}$ whenever $2^{n-1}\geq l(X)+1$. It is clear that \cref{FI-str-Lie-solv} implies \cref{DI-Lie-solv}, \cref{DI-Lie-solv} implies \cref{FI-pol-ident}, and \cref{U(FI)-solv} obviously implies \cref{U(FI)-poly-ident}. The fact that \cref{U(FI)-poly-ident} implies \cref{X-bounded} can be proved by means of the same proof as in \cite[Lemma 2]{Spiegel2002}. Therefore, in order to complete the proof it is enough to prove that \cref{FI-str-Lie-solv} implies \cref{U(FI)-solv}. For this purpose, apply \cref{dl(G)-and-dl^Lie(I)} with $I=FI(X,K)$ and $G=\U(FI(X,K))$.
		% we will show by induction that, for every $n\geq 0$, the $n$-th term $\U^{(n)}$ of the derived series of $\U(FI(X,K)))$ is contained in $\delta+FI(X,K)^{(n)}$. The claim is trivial for $n=0$. Assume then $n>0$ and $\U^{(n-1)}\subseteq \delta+FI(X,K)^{(n-1)}$. For all $\alpha,\beta \in \U^{(n-1)}$ we have $(\alpha,\beta)=\delta+[\alpha^{-1},\beta^{-1}]\alpha\beta\in \delta+FI(X,K)^{(n)}$. Since $\delta+FI(X,K)^{(n)}$ is a monoid,  this yields the desired conclusion, which completes the proof. 
	\end{proof}

	\begin{rem}\label{remark} \emph{
			Let $(X,\leq)$ be a bounded poset and $K$ a field. It follows from the proof of Theorem \ref{teor} that $FI(X,K)$ does not satisfy any polynomial identity of degree less than $2l(X)+2$.
		}
	\end{rem}
	
	Some consequences of Theorem \ref{teor} are given. The next result provides the exact value of the Lie derived length of a Lie solvable finitary incidence algebra and of the derived length of $\U(FI(X,K))$ under a mild hypothesis on the ground field. As usual, we denote by $\lceil t \rceil$ the upper integral part of the real number $t$. Moreover, for a prime $p$, the field with $p$ elements is denoted by $\F_p$. 
	
	\begin{cor}\label{cor1} Let $(X,\leq)$ be a poset and $K$ a field. 
		\begin{enumerate}
			\item\label{dl-FI} If $FI(X,K)$ is Lie solvable, then 
			$$
			\dl_{Lie}(FI(X,K))=\dl^{Lie}(FI(X,K))=\lceil \log_2 (l(X)+1) \rceil + 1.
			$$
			\item\label{dl-U(FI)} If $\U(FI(X,K))$ is solvable and $K \neq \F_2$, then 
			$$\dl(\U(FI(X,K)))= \lceil \log_2 (l(X)+1) \rceil + 1.$$
		\end{enumerate}
	\end{cor}
	\begin{proof} %Put $A=I(X,R)$. 
		\cref{dl-FI}. As $FI(X,K)$ is Lie solvable, $X$ is bounded by \cref{teor}. Therefore, by \cref{ideals-Z_k}\cref{FI^(n)-sst-Z_(2^(n-1))} we have  $FI(X,K)^{(n)}=\{0\}$ whenever $2^{n-1}\geq l(X)+1$, hence $\dl^{Lie}(FI(X,K))\leq \lceil \log_2 (l(X)+1) \rceil + 1$. On the other hand, as a Lie solvable algebra of derived length $m$ satisfies a polynomial identity of degree $2^m$, it follows from \cref{remark} that $FI(X,K)^{[n]}\neq \{0\}$ for every $n$ such that $2^{n-1}< l(X)+1$. Hence $\dl_{Lie}(FI(X,K))\ge \lceil \log_2 (l(X)+1) \rceil + 1$, yielding the first part of the statement. 
		
		\cref{dl-U(FI)}. 
		% It was shown in the proof of \cref{teor} that the $n$-th term $\U^{(n)}$ of the derived series of $\U(FI(X,K))$ is contained in $\delta+FI(X,K)^{(n)}$, which forces $\dl(\U(FI(X,K))) \leq \lceil \log_2 (l(X)+1) \rceil + 1$.
		We have $\dl(\U(FI(X,K))) \leq \lceil \log_2 (l(X)+1) \rceil + 1$ by \cref{dl(G)-and-dl^Lie(I),dl-FI}.
		Let $C$ be a chain of $X$ with $\vert C \vert=l(X)+1$.  Since $\vert K \vert >2$, it follows from \cite[Theorem 22 and Lemma 28(g)]{DHR} applied to $FI(C,K)$ that $\U(FI(X,K))^{(n)}\supseteq \U(FI(X,C))^{(n)}\neq \{\delta\}$ whenever $2^{n-1}< l(X)+1$. We conclude that $\dl(\U(FI(X,K)))\geq \lceil \log_2 (l(X)+1) \rceil +1$, and the result follows. 
		% Choose $a\in K\setminus\{-1\}$, so that $1+ae_x\in FI(X,K)^\times$ for all $x\in X$. Then $(1+ae_x,1+e_{xy})=1+ae_{xy}\in U^{(1)}$ for all $x<y$. Moreover, $(1+ae_{xy},1+e_{yz})=1+ae_{xz}$ for all $x<y<z$, and the same way as in \cite[Lemma 28(g)]{DHR} one proves that $1+ae_{xy}\in U^{(n)}$ whenever $2^{n-1}< l(x,y)$. We conclude that $\dl(FI(X,K)^\times)\geq \lceil \log_2 l(X) \rceil +1$, completing the proof.  
	\end{proof}
	
	\begin{rem}\label{remark1}\emph{
			The ground field $\F_2$ was correctly omitted in  \cref{cor1}\cref{dl-U(FI)}. For instance, if $X=\{1,2,\ldots,n\}$ with  the usual order, then  $\U(FI(X,\F_2))$ is isomorphic to the group $UT(n,\F_2)$ of all $n\times n$ upper unitriangular matrices over $\F_2$. It is well-known that the derived length of $UT(n,\F_2)$ is $\lceil \log_2 n \rceil$.}
	\end{rem}
	
	Following \cite{DHR}, given a poset $(X,\leq)$ and a field $K$,  we set
	$$
	\U_1=\delta+J(FI(X,K))=\{f\in FI(X,K) \mid \alpha_{xx}=1 \textrm{ for all } x\in X\}.
	$$
	
	Solvability and nilpotency of $\U_1$ were characterized in \cite[Theorem 41]{DHR}. In our next result we provide an alternative, short proof of the just mentioned theorem, together with the exact value of the nilpotency class and the derived length of $\U_1$:
	
	\begin{cor}\label{cornilp}
		Let $(X,\leq)$ be a poset and $K$ a field. Then the following conditions are equivalent:
		\begin{enumerate}
			\item\label{U_1-solv} $\U_1$ is solvable;
			\item\label{U_1-nilp} $\U_1$ is nilpotent;
			\item\label{X-bounded-for-U_1} $X$ is bounded.
		\end{enumerate} 
		Moreover, in this case $\U_1$ has nilpotency class $l(X)$ and derived length $\lceil \log_2 (l(X)+1) \rceil$. 
	\end{cor}
	\begin{proof}
		Suppose that $\U_1$ is solvable. Note that $\U_1$ is isomorphic to the adjoint group $(J(FI(X,K)),\circ)$ of $J(FI(X,K))$, where $f\circ g=f+g+fg$ for all $f,g\in J(FI(X,K))$. Therefore, by \cite[Theorem A]{AS} we have that 
		$J(FI(X,K))$ is Lie solvable. As a consequence, by \cref{ideals-Z_k}\cref{FI^(1)-sst-Z_1} we deduce that $FI(X,K)$ is Lie solvable, and so \cref{teor} implies that $X$ is bounded. Hence \cref{U_1-solv} implies \cref{X-bounded-for-U_1}, and \cref{U_1-nilp} obviously implies \cref{U_1-solv}. Now, if $X$ is bounded, then $J(FI(X,K))$ is Lie nilpotent by \cref{ideals-Z_k}\cref{gamma_n(J(FI))}, so that $\U_1$ is nilpotent by \cite[Corollary]{D}. Therefore the three conditions of the statement are equivalent.
		
		Finally, it follows from \cref{ideals-Z_k}\cref{gamma_n(J(FI))} and \cref{remark} that the Lie nilpotency class of $J(FI(X,K))$ is $l(X)$. By \cite[Corollary]{D}, this value coincides with the nilpotency class of $\U_1$. By applying \cref{dl(G)-and-dl^Lie(I)} to $I=J(FI(X,K))$ and $G=\U_1$, we obtain that $\dl(\U_1)\le\dl^{Lie}(J(FI(X,K)))$. It follows from \cref{ideals-Z_k}\cref{J^(n)-sst-Z_(2^n)} that $\dl^{Lie}(J(FI(X,K)))\leq \lceil \log_2 (l(X)+1) \rceil$. Now, if $K\neq \F_2$ then note that $\U(FI(X,K))^{(1)}\su\U_1$, so $\dl(\U_1)\ge\dl(\U(FI(X,K)))-1=\lceil \log_2 (l(X)+1) \rceil$. On the other hand, if $K= \F_2$, then we have $\U(FI(X,K))=\U_1$. Take a chain $C$ in $X$ with $\vert C \vert=l(X)+1$. Then \cref{remark1} allows to conclude that $\dl(\U_1)\ge\dl(\U(FI(C,K)))=\lceil \log_2 (l(X)+1) \rceil$, and the assertion follows. 
		% Since $\U(FI(X,K))^{(1)}\su\U_1$, then $\dl(\U_1)\ge\dl(\U(FI(X,K)))-1=\lceil \log_2 (l(X)+1) \rceil$. On the other hand, applying \cref{dl(G)-and-dl^Lie(I)} to $I=J(FI(X,K))$ and $G=\U_1$, we obtain and $\dl(\U_1)\le\dl^{Lie}(J(FI(X,K)))$. It follows from \cref{FI^(1)-sst-Z_1,Z_mZ_n-sst-Z_(m+n)} of \cref{ideals-Z_k} that $J(FI(X,K))^{(n)}\su Z_{2^{n+1}-1}\su Z_{2^n}$ for all $n\ge 0$, yielding $\dl^{Lie}(J(FI(X,K)))\leq \lceil \log_2 (l(X)+1) \rceil$.
	\end{proof}   
	
	Since $\U(FI(X,\F_2))=\U_1$, we have the following result complementing \cref{cor1}\cref{dl-U(FI)}.
	
	\begin{cor}\label{dl(U)-for-K=F_2}
		If $\U(FI(X,\F_2))$ is solvable, then 
		$$\dl(\U(FI(X,\F_2)))= \lceil \log_2 (l(X)+1) \rceil.$$
	\end{cor}
	
	We conclude with the following
	
	\begin{cor}\label{nilpotency-of-FI} Let $(X,\leq)$ be a poset and $K$ a field. Then 
		\begin{enumerate}
			\item\label{FI-Lie-nilp} $FI(X,K)$ is Lie nilpotent if and only if $X$ is an antichain;
			\item\label{U(FI)-nilp} $\U(FI(X,K))$ is nilpotent if and only if either $X$ is antichain or $K=\F_2$ and $X$ is bounded. 
		\end{enumerate}
	\end{cor}
	\begin{proof}
		\cref{FI-Lie-nilp}. If $X$ is an antichain, then $FI(X,K)$ is commutative, hence Lie nilpotent. Conversely, suppose that $X$ is not an antichain and let $x,y\in X$ with $x<y$. Then $[e_{xy},e_{yy}]=e_{xy}$, hence $FI(X,K)$ is not Lie nilpotent.
		
		\cref{U(FI)-nilp}. Suppose that $\U(FI(X,K))$ is nilpotent. Then $X$ is bounded by \cref{teor}. Assume $K \neq \F_2$. If $X$ were not an antichain, then we could find a chain $C$ in $X$ with $\vert C \vert=n>1$. As $\U(FI(C,K))$ is isomorphic to the group of all invertible $n\times n$ upper triangular matrices over $K$, it follows that $\U(FI(X,K))$ is not nilpotent, a contradiction. This proves necessity. 
		
		Let us prove sufficiency. Clearly, $\U(FI(X,K))$ is abelian when $X$ is an antichain. On the other hand, if $K=\F_2$, then $\U(FI(X,K))=\U_1$, which is nilpotent by \cref{cornilp} provided that $X$ is bounded.
	\end{proof}
	
	\section*{Acknowledgements}
	
	Mykola Khrypchenko was partially supported by CMUP, member of LASI, which is financed by national funds through FCT --- Fundação para a Ciência e a Tecnologia, I.P., under the project with reference UIDB/00144/2020. Salvatore Siciliano is a member of the "National Group for Algebraic and Geometric Structures, and their applications" (GNSAGA-INdAM). The authors are grateful to the referee for pointing out misprints and other helpful comments that improved the paper.
	
	\bibliography{bibl}{}
	\bibliographystyle{acm}

\end{document}